\documentclass{article}
\usepackage{mathrsfs}
\usepackage{graphicx}
\usepackage{amsmath}
\usepackage{amsthm}
\usepackage{amssymb}

\newtheorem{thm}{Theorem}
\newtheorem{lem}{Criterion}
\newtheorem{Defi}{Definition}

\begin{document}

\title{The Global Convergence Analysis of the Bat Algorithm
 Using a Markovian Framework and Dynamical System Theory}

\author{Si Chen$^{1}$, Guo-Hua Peng$^1$ \\
Department of Mathematics, College of Science, \\
Northwestern Polytechnical University, Xi'an 710072, P. R. China. \\[10pt]
Xingshi He$^2$ \\
College of Science, Xi'an Polytechnic University, \\
No. 19 Jinhua South Road, Xi'an 710048, P. R. China \\[10pt]
Xin-She Yang$^3$ \\
School of Science and Technology, Middlesex University, \\
London NW4 4BT, UK.
}

\date{}

\maketitle

\begin{abstract}
The bat algorithm (BA) has been shown to be effective to solve a wider range of optimization problems. However, there is not much theoretical analysis concerning its convergence and stability. In order to prove the convergence of the bat algorithm, we have built a Markov model for the algorithm and proved that the state sequence of the bat population forms a finite homogeneous Markov chain, satisfying the global convergence criteria. Then, we prove that the bat algorithm can have global convergence. In addition, in order to enhance the convergence performance of the algorithm, we have designed an updated model using the dynamical system theory in terms of a dynamic matrix, and the parameter ranges for the algorithm stability are then obtained. We then use some benchmark functions to demonstrate that BA can indeed achieve global optimality efficiently for these functions. \\[10pt]


\end{abstract}

\noindent {\bf Citation Details:} {\it Si Chen, Guo-Hua Peng, Xing-Shi He, Xin-She Yang, Global convergence analysis of the bat algorithm using a Markovian framework and dynamical system theory, Expert Systems with Applications, vol. 114, 173--182 (2018). }
https://doi.org/10.1016/j.eswa.2018.07.036

\section{Introduction}
With the development of computational intelligence~\cite{Engelbrecht,Yang,Jain,Akerkar}, nature-inspired algorithms have been shown to be effective and thus become widely used for various optimization problems~\cite{Kennedy,Koziel,Yang}. However, there is still a significant gap between theory and practice. Though the applications of algorithms are very successful, the relevant fundamental theory lacks behind or no theory at all.
For example, the bat algorithm (BA), developed by Xin-She Yang in 2010 \cite{Yang01,Yang02}, has been shown to very efficient in  practice, but there is no mathematical theory for analyzing this algorithm.  In fact, most of the swarm intelligence based algorithms for computational intelligence have no or little theoretical analyses, except for a few algorithms, such as the well known particle swarm optimization~\cite{Clerc,Jiang,Eberhart,Ren} and genetic algorithms~\cite{Deb,Leung}. Though we know these algorithms can work well in practice, we rarely understand why they work so well and under what conditions or parameter ranges. These key challenges require further in-depth theoretical studies.

Recent studies started to focus on this area and some preliminary results have been obtained. Sheng et al.~analyzed the convergence of BA algorithm according to the global convergence criterion of stochastic optimization algorithm \cite{Sheng}, while Li et al.~defined the two modes of speed and position updating for the bat algorithm, followed by the analysis of the two modes defined by the characteristic equation~\cite{ZLi}. Huang et al.~constructed a class of globally convergent BA algorithm to prove its global convergence, while applying it to solve large-scale optimization problems~\cite{Huang}. However, these studies have focused on the modified version of the bat algorithm, and there still lacks rigorous results about the standard bat algorithm, concerning both its convergence and stability.

Therefore, the main aim of this work is to prove the convergence of the standard bat algorithm using both Markov chain theory and then dynamic matric so as to gain insight into the working mechanisms of this algorithm. The paper is thus organized as follows. We will first introduce the basics of the BA in Section 2, followed by the introduction of the global convergence criteria of random search algorithms in Section 3. We will then build a proper Markov model for BA, and outline the main steps of proof of convergence in Section 4. A further updating model of the bat algorithm is defined in Section 5 so as to study the parameter variations and convergence conditions for the algorithm. The numerical experiments of some selected benchmark functions and their convergence behaviour are then presented in Section 6. Finaly, we draw sone brief conclusions in Section 7.

\section{Standard Bat Algorithm}
The standard bat algorithm was developed by Yang in 2010 \cite{Yang01} to solve continuous optimization problems. It has been extended to multiobjective optimization \cite{Yang02} with many different applications \cite{Gandomi,Natarajan,Yang03}.

The BA algorithm, inspired by the echolocation behavior of microbat species, is a population-based
algorithm using the frequency tuning with varying pulse emission rates and loudness so as to mimic the main nature of bats' echolocation when hunting for prey. The intention of BA is to act as a global optimizer using sufficient randomization and autoswitching between local and global moves, controlled by the actual emission rates and loudness of individuals \cite{Yang01}.

Based on the original bat algorithm \cite{Yang01}, each bat has a position vector $x_i^t$ and a flying velocity $v_i^t$ at iteration $t$ in a $d$-dimensional search space. Their main algorithmic equations can be written as follows:
\begin{eqnarray}
&& f_i=f_{min}+(f_{max}-f_{min})\beta, \\
&& v_i^{t+1}=\omega v_i^t+(p-x_i^t)f_i, \\
&& x_i^{t+1}=x_i^t+v_i^{t+1},
\end{eqnarray}
where $f_i$ is the acoustic frequency of the $i$-th bat in the range of [$f_{\min}, f_{\max}$]. Here, $\omega$ indicates the inertia weight in the update of velocity, and in the standard bat algorithm $\omega=1$ was used. For generality, we can use $\omega \in (0, 1)$. In addition, $\beta \in [0,1]$ is a uniformly distributed random variable and $p$ corresponds to the current best solution found by all the bats.

In a local search move, a new solution will be generated randomly around the old solution, often the current best solution. That is
\begin{eqnarray}
\textsl{X}_{new}=\textsl{X}_{old}+\varepsilon \textsl{A}^t,
\end{eqnarray}
where $\textsl{X}_old$ is a solution chosen from the current best solution set, $\textsl{A}^t$ is the mean of the bats' loudness at $t$, and $\varepsilon$ is the random number in $[-1,1]$. The loudness $A_i$ and the velocity $r_i$ of the bats can be updated as follows:
\begin{eqnarray}
\textsl{A}_i^{t+1}=\alpha \textsl{A}_i^t, \quad r_i^{t+1}=r_i^0[1-\exp(-\gamma t)],
\end{eqnarray}
where $\alpha$ and $\gamma$ are constants ($0<\alpha<1, 0<\gamma$).

For the convenience of the discussions below, we can rewrite the above equations (1) to (3) in the following form:
\begin{eqnarray}
v_{t+1}=\omega v_t+(p-x_t) w_i,
\end{eqnarray}
which is valid for each individual bat. Here, we use $w_i$ to denote frequencies to avoid potential confusion with the objective function $f(x)$ later. The position vectors can be updated iteratively as
\begin{eqnarray}
x_{t+1}=(1+\omega-f_{i,t})x_t+p w_i
\end{eqnarray}
The stopping condition is usually  the maximum number of iterations, or when the optimal value searched by the population satisfies the set minimum fitness value.

The bat algorithm has many variants with a diverse range of applications \cite{Akhtar,Natarajan,Gandomi,Baziar,Bora,Ramesh,Khan,Mishra,Kang,Tsai,Jaddi,Niknam,Yang01,Yang02}, such as fuzzy clustering, prediction, classification, image processing, feature selection, scheduling, and data mining.

\section{Convergence Criteria}
Depending on the actual algorithms and the framework of theoretical analysis, the convergence of an algorithm can be tested by certain criteria. One of the commonly used criteria is based on the two conditions outlined by Solis and Wets \cite{Solis}.

Let $\langle H,f\rangle$ be the optimization problem with a fitness function and a feasible solution space $H$. A stochastic optimization $S$ iterates for $t$ iterations and the new solution $x_{t+1}$
can be obtained from solution $x_t$ by
\begin{equation}
 x_{t+1} =S (x_t, \zeta),
\end{equation}
where $\zeta$ is the solution set found by the algorithm $S$ during the iterative process.

Let us define the bounds of the search on the Lebesgue metric space as the infinum
\begin{eqnarray}
\theta=\inf\{k:v[x \in H | f(x)<k]>0\},
\end{eqnarray}
where $v[X]$ is the measure on set $X$, which means that there are non-empty subsets in the search space and the fitness value corresponding to the element in the non-empty subset can be infinitely close to $\theta$. Thus, the neighbourhood or region of optimal solutions can be defined as
\begin{eqnarray}
R_{\varepsilon ,M}=\begin{cases}
\{x\in H|f(x)<\theta+\varepsilon\},& -\infty<\theta<\infty, \\[5pt]
\{x\in H|f(x)<M\},& \theta=-\infty,
\end{cases} \end{eqnarray}
where $\varepsilon >0$ and $M<0$. If a stochastic algorithm finds a point in $R_{\varepsilon, M}$, then we can consider that the algorithm finds the global optimal solution or an approximation to the global optimal solution.

In general, for two settings, two conditions are necessary to guarantee the global optimality is achievable:

Condition 1: An optimization algorithm $S$ should guarantees that the sequence  $\{f(x^t)\}_{t=0}^\infty$ is decreasing. If there is $f(S(x,\zeta))\leq f(x)$ , and at the same time there is established $\zeta\in H$ , we have
\begin{eqnarray}
f(S(x,\zeta))\leq f(\zeta).
\end{eqnarray}

Condition 2: For all subsets $\forall B \in H$ subject to $v(B)>0$,
we have
\begin{eqnarray}
\prod_{t=0}^\infty (1-u(B))=0,
\end{eqnarray}
where $u_t(B)$ represents the probability measure of the $t$-th iterative result of the random algorithm $S$ on $B$.

Mathematically speaking, a stochastic optimization algorithm that can have a guaranteed global convergence is based on the following lemma or criterion \cite{Solis,Jiang,Villa}

\begin{lem}
 For $f$ is measurable and the feasible solution space $H$ is a measurable subset on $R^n$, if the stochastic algorithm $S$ satisfies both Condition 1 and Condition 2, sequence $\{x_t\}_{t=0}^\infty$ is generated by the algorithm $S$ will lead to
\begin{eqnarray}
\lim_{t\rightarrow \infty} P(x_t \in R_{\varepsilon,M}) = 1,
\end{eqnarray}
where $P(x_t \in R_{\varepsilon,M})$ represents the probability that the best solution obtained by algorithm $S$ after $t$ iterations belongs to $R_{\varepsilon,M}$.
\end{lem}

In other words, the above criterion means that the algorithm will converge with a probability one as the number of iterations is sufficiently large, which equivalently means that the algorithm can have almost guaranteed global convergence.

\section{Global Convergence Analysis}

In order to prove the convergence of the bat algorithm, we will introduce some preliminaries.
If the position of each bat individual in the BA algorithm is considered as a state $x$, then the process of states $x_t$ with pseudotime or iteration counter $t$ can be considered as a random process. For such a stochastic process, the Markov chain can be an effective tool to analyze its convergence in a probability sense.

\subsection{Preliminaries}
Let us first define the state, state space and other relevant concepts that will be later used for proving the global convergence of the BA.

The states of bats and the state space can be defined as follows:
\begin{Defi}
The position of a bat individual $x$ with velocity $v$ and historical best position $p$ forms its state or status, denoted by $a=(x,v,p)$, where $x, p \in H$.
\end{Defi}
In addition, we have $f(p)\leq f(x)$ and $v\in[v_{min},v_{max}]$. All possible states of all bats form a state space for bats, denoted by
\begin{eqnarray}
A=\Big\{{a=(x,v,p)|x,p \in H, \; f(p)\leq f(x), \; v \in [v_{min},v_{max}]}\Big\}.
\end{eqnarray}

Furthermore, the states and state space of the bats population or group can be defined as follows:
\begin{Defi}
The set of all $N$ bat individuals is called the bat group, and the states of this bat group can be denoted by $b=(a_1, a_2, \ldots, a_N)$. The collection of all possible bat group status or states forms the bat group status space, denoted by
\begin{eqnarray}
B=\Big\{b=(a_1, a_2, \dots, a_N), a_i \in A (1\leq i \leq N)\Big\}.
\end{eqnarray}
\end{Defi}

From the above definitions, it is obvious that the definition of bat group status $B$ already contains the best position (or the best solution vector) in the group history. Furthermore, the state transition for the positions of bats representing solutions can be defined as follows:

For $\forall a_1=(x_1,v_1,p_1)\in A$ and $\forall a_2=(x_2,v_2,p_2)\in A$ during the iterations of the BA algorithm, the state transition from $a_1$ to $a_2$ can be denoted by
\begin{eqnarray}
F_A(a_1)=a_2,
\end{eqnarray}
where $F_A$ is the transition function from $a_1$ to $a_2$ in the state space $A$.

Similarly,   for $\forall b_i=(a_{i,1}, a_{i,2}, \ldots, a_{i,N})\in H$ and $\forall b_j=(a_{j,1}, a_{j,2}, \ldots, a_{j,N}) \in H$, the iterative process of the BA algorithm in essence transfers the bat group states from $b_i$ to $b_j$. That is
\begin{eqnarray}
F_b(b_i)=b_j.
\end{eqnarray}

\subsection{Markov Chain Model for BA}

In order to prove the convergence using a Markov chain framework, we have to build a Markov chain model for the bat algorithm. Let us first start with a theorem:
\begin{thm}
In the BA algorithm, the bat status $a_1$ is essentially shifted in one step to the status $a_2$, and its transition probability is the joint probability
\begin{eqnarray}
P(F_A(a_1)=a_2)=P(x_1\rightarrow x_2) P(v_1\rightarrow v_2) P(p_1\rightarrow p_2),
\end{eqnarray}
where $P(x_1\rightarrow x_2)$ is the transition probability of the bat position
from $x_1$ to $x_2$, $P(v_1\rightarrow v_2)$ is the transition probability of the bat velocity from $v_1$ to $v_2$, and $P(p_1 \rightarrow p_2)$ is the transition probability of the best position (in the whole history) from $p_1$ to $p_2$.
\end{thm}

\begin{proof}
The status of a bat is transferred via $a_1(x_1, v_1, p_1) \rightarrow a_2(x_2, v_2, p_2)$. That is, $x_1\rightarrow x_2$, $v_1\rightarrow v_2$, and $p_1\rightarrow p_2$ are carried out simultaneously. The joint probability of $F_A(a_1)\rightarrow a_2$ is
\begin{eqnarray}
P(F_A(a_1)=a_2)=P(x_1\rightarrow x_2) P(v_1\rightarrow v_2) P(p_1\rightarrow p_2).
\end{eqnarray}

From the updating equations for velocities and positions (see Eqs.(2) and (3)), it is easy to see that the transition probability of the positions of bats can be calculated by
\begin{eqnarray}
P(x_1\rightarrow x_2)=\begin{cases}
\frac{1}{|f(p_g-x_1)|},
& v_2\in [x_1+\omega v_1,\omega v_1+w_i(p_g-x_1)], \\[5pt]
0, & v_2\notin [x_1+\omega v_1,\omega v_1+w_i(p_g-x_1)].
\end{cases}
\end{eqnarray}
Similarly, the transition probability concerning the velocities of bats can be calculated by
\begin{eqnarray}
P(v_1\rightarrow v_2)=\begin{cases}
\frac{1}{|f(p_g-x_1)|}, & v_2\in [\omega v_1,\omega v_1+w_i(p_g-x_1)], \\[5pt]
0, & v_2 \notin [\omega v_1,\omega v_1+ w_i(p_g-x_1)].
\end{cases}
\end{eqnarray}
In addition,  the transition probability of the best position $p$ of all bats is
\begin{eqnarray}
P(p_1\rightarrow p_2)=\begin{cases}
1, & f(p_2)<f(p_1), \\[5pt]
0, & f(p_2)\geq f(p_1).
\end{cases}
\end{eqnarray}
It is worth pointing out that we treat the optimization problem as a minimization problem. Thus, $p_2$ is better than $p_1$ if $f(p_2)<f(p_1)$.
\end{proof}

With these results, we can now prove the following theorem:
\begin{thm}
In the iterative process of the BA algorithm, the transition probability of the bat group status $b_i$ to $b_j$ is given by
\begin{eqnarray}
P(F_b(b_i)=b_j)=\prod_{t=1}^N P(F_A(a_{it})=a_{jt}),
\end{eqnarray}
where $N$ is the total number of iterations so far.
\end{thm}

\begin{proof}
 As $F_b(b_i)=b_j$ indicates that each state in the bat group state, $b_i$ is simultaneously transferred to group state $b_j$; that is \[ F_A(a_{i1})=a_{j1}, F_A(a_{i2})=a_{j2}, \ldots, F_A(a_{iN})=a_{jN}. \]
Then, the transition probability of a group transition of the bat group should
be the joint probability of each iteration step. Thus, we have
\begin{eqnarray} \begin{split}
P(F_b(b_i)=b_j)&=P(F_A(a_{it})=a_{jt})P(F_A(a_{it})=a_{jt}) \ldots P(F_A(a_{iN})=a_{jN})\\
&=\prod _{t=1}^NP(F_A(a_{it})=a_{jt}),
\end {split} \end{eqnarray}
which concludes the proof.
\end{proof}

Now we have to show that the state sequence $a$ is finite, homogeneous Markov chain.
\begin{thm}
In the BA algorithm, the bat group state sequence $a$ is indeed a finite homogeneous Markov chain.
\end{thm}

\begin{proof}
For any optimization algorithm, its search space during the whole iterative process is finite because both the population size and the number of iterations are finite, so each of the bat state $a=(x, v, p)$ among the $x, v, p$ are finite, which leads to the fact that the bat state space is finite.

From the algorithmic equations outlined in Section 2, the position updates of each bat individual is a second-order equation, so the random process of positions of the BA algorithm changes with time, which is not the Markov process. However,
if we can group the position, velocity and global optimal values together as one state $B$, then state $B(t+1)$ is only related to state $B(t)$, not its history. Then, sequence $B$ has proper Markov chain properties.

From Eqs.(1)-(3), $\beta \in [0,1]$ is a random vector that is uniformly distributed, and the algorithmic equations [i.e., (1) to (3)] form a stochastic
system. It is straightforward to show that the state $B(t-1)$  of the system at time $t$ transferring to the new state $B(t)$ is completely determined by its state at time $t$. In addition, the factor $\gamma$ and $\omega$ as well as the pseudotime $t$ in the iterative formulas are independent of the state of the system before time $t$.

From $B(t-1)$ to $B(t)$ of bats group state sequence $\{B(t); t\geq o\}$, the transition probability $P(F_B(B(t-1))=B(t))$ of the two states is determined by the transition probability of all individuals in the bats group, and the probability of transition can be calculated by the joint probability
 of $P(x(t-1)\rightarrow x(t))$, $P(v(t-1)\rightarrow v(t))$, and $P(p(t-1)\rightarrow p(t))$, according to Theorem 1.

In addition, $P(x(t-1)\rightarrow x(t))$ and $P(v(t-1)\rightarrow v(t))$ are only related to $x, v, p$ at time $t-1$. Thus, $P(F_B(B(t-1))=B(t))$ is only related to the state $a_i(t-1), 1\leq i \leq N$ of all bats at time $t-1$. Therefore, the  Markov chains are finite.

Furthermore, from Theorem 1, $P(F_A(a(t-1))=a(t))$ is independent of time $t-1$.
Similar argument also indicates that $P(F_B(B(t-1))=B(t))$ is also independent of $t-1$. Therefore, the finite Markov chains are homogeneous.
\end{proof}

\subsection{Global Convergence of the BA }
With the above definitions and theorems, let us proceed to prove the convergence of the bat algorithm.

For the true optimal solution $g$ for an optimization problem $\langle H, f\rangle$ with an objective function $f(x)$ where $x$ is a vector, the optimal state set can be defined as
\begin{equation}
L=\{a=(x,v,p) |f(p)=f(g), a \in A\}.
\end{equation}
Obviously we have $L \subseteq A$ as $L$ should be a subset of $A$. If in any case
$L=A$, any solution in $A$ is equally optimal, which means that objective landscape is flat (thus it is equivalent to a feasibility problem in which the objective does not exert any selection pressure on different solutions). This is just a special case and the optimal solution is already achieved, and thus we will not discuss this case any further.

In addition, for the optimal solution $g$ to an optimal problem
$\langle H, f\rangle$, the optimal bat group state set can be defined as
\begin{equation}
 U=\{B=(a_1, a_2, \dots, a_N)\big| \exists a_i \in L, 1\leq i\leq N\},
\end{equation}
which means that the optimal bat groups state set $U$ is the set of all bat groups such that at least one bat individual in the population
with its state belong to $L$.

Using the same methodology as outlined in \cite{Hexs} (Theorems 7 and 8 in their paper) and the results in \cite{Wen-xiu}, we can prove the following three theorems:
\begin{thm}
When $U\subset B$, there is no closed set $I$ other than $B$ such that $I \bigcap U=\emptyset$.
\end{thm}
\begin{thm}
 If a Markov chain has a non-empty set $Z$ with no closed set $D$ satisfying $Z\bigcap D=\emptyset$, then $\lim_{t\rightarrow \infty}P(x_t=j)=\pi_j$, only if $j \in Z$, and $\lim_{t\rightarrow \infty} P(x_t=j)=0$ only if $j \notin Z$.
\end{thm}
\begin{thm}
If the number of iteration approaches infinity or sufficiently large, the group state sequence will converge to the optimal state set $U$.
\end{thm}

From the above four theorems, it is straightforward to prove the following global convergence theorem:
\begin{thm}
The bat algorithm with the Markov chain model defined in Section 4.2
has guaranteed global convergence.
\end{thm}

\begin{proof}
From the convergence criterion (Criterion 1), we know that if a stochastic optimization algorithm can satisfy both condition 1 and condition 2, it will converge to global optimality. In essence, the first condition (Condition 1) can guarantee that the fitness value $f(x)$ of the stochastic optimization algorithm is decreasing. From the above discussions, we know that, in the iterative process of the BA algorithm, it is obvious that $f(p) \leq f(x_t)$

Furthermore, the previous theorem means that the group state sequence will converge towards the optimal set after a sufficiently large number of iterations, which means that the probability of not reaching the globally optimal solution is asymptotically zero. This means that the second convergence condition is also satisfied. As a result, a conclusion can be drawn that BA has guaranteed global convergence towards its global optimality with a probability one.
\end{proof}

This proof is based on the Markov chain framework, and thus the convergence is in a probabilistic sense. It is an important result because it shows that the bat algorithm can indeed converge. However, there is no information about the rate of convergence and how the parameters may affect the convergence behaviour of this algorithm.

It is worth pointing out that the above proof has been based on a simplified Markov model for the bat algorithm. The standard bat algorithm also includes the variation of pulse emission rate and loudness, which has not been considered here. However, the overall convergence behaviour can be very similar.

In order to gain further insight into the parameter values and their effect on the convergence of the bat algorithm, we now use a completely different approach
to analyze the algorithm in terms of dynamic matrix theory.

\section{Convergence Analysis Based on Dynamic Matrix Theory}

The advantage of algorithm analysis using dynamic matrix theory is that the matrix can be constructed from the updating equations of the algorithm, and the insight can be gained about the possible parameter ranges for the algorithm to converge.

For this purpose and for simplicity of calculations without losing generality,
it is assumed that the current optimal solution in the bat algorithm population is  a constant vector $p$ (even though it is updated at each iteration). It is assumed that the frequency $f_i$ is a constant $m \geq 0$. Within this framework, the velocities and positions of bats during the iterations can be written as
\begin{equation}
v_{k+1}=l v_k+(p-x_k) m, \label{BA-equ-100}
\end{equation}
\begin{equation}
x_{k+1}=c x_k+uv_{k+1}, \label{BA-equ-200}
\end{equation}
where coefficient $m$ is essentially the average of the frequencies, while $l$, $c$ and $u$ are the weight coefficients so that we can analyze the algorithm in general.  The attraction point $p$ in the $d$-dimensional space is the current optimal position. The algorithm represented by the system (\ref{BA-equ-100})
and (\ref{BA-equ-200})
now have four parameters to be tuned. They are $l, m, c, u$.
We will show that two of these parameters are key parameters.

\subsection{Dynamic Matrix Model for the Bat Algorithm}

From the algorithmic equations (\ref{BA-equ-100}) and (\ref{BA-equ-200}), we can rewrite (\ref{BA-equ-100}) equivalently using the previous iteration as
\[ x_k = c x_{k-1} + u v_k, \]
and then multiply its both sides by $l$ and re-arrange slightly, we have
\begin{eqnarray}
l u v_k=l x_k-c l x_{k-1}. \label{BA-equ-300}
\end{eqnarray}
Combining (\ref{BA-equ-100}) and (\ref{BA-equ-200}), we have
\[ x_{k+1}= c x_k + u v_{k+1} = c x_k + u [l v_k +(p-x_k) m] \]
\begin{equation}
=c x_k + u l v_k + u m p - u m x_k
=c x_k + [l x_k - c l x_{k-1} ] + u m p - u m x_k,
\end{equation}
where we have used Eq.(\ref{BA-equ-300}).

By re-arranging the above equation, we have a recursive relationship for $x_k$
\begin{eqnarray}
m u p=x_{k+1}+(m u-c-l)x_k + l c x_{k-1}. \label{BA-equ-400}
\end{eqnarray}
It is obvious that $mu$ always appear as a factor, not individually. This means that only their product matters. Thus, for simplicity (without loss of generality), we can set
\begin{eqnarray}
u \equiv 1.
\end{eqnarray}
Therefore, we have a reduced system of algorithmic equations for the bat algorithm as \begin{eqnarray} v_{k+1} & = & lv_k+m(p-x_k), \\
x_{k+1} & = & c x_k+ v_{k+1}. \end{eqnarray}

In addition, as the number of iterations $k$ increases, it can be expected that the series should converge to $p$ (the current best solution found by all the bats). That is
\begin{equation}
\lim_{k \rightarrow \infty} x_k =p, \quad \lim_{k \rightarrow \infty} x_{k+1} =p,
\quad \lim_{k \rightarrow \infty} x_{k-1} =p.
\end{equation}
Taking the limit of (\ref{BA-equ-400}) and using the above results, we have
\begin{equation}
mp = p + (m-c-l) p + l c p,
\end{equation}
which gives that
\begin{equation}
p (l-1) (c-1)=0.
\end{equation}
Thus, either $p=0$ (a trivial solution or special solution), or $l=1$, or $c=1$. Considering the role of $c$ in the algorithm, we can set $c=1$ for the moment as it does not affect the update of the position vectors.

Now we have finally obtained the reduced dynamic system for the bat algorithm
\begin{eqnarray} v_{k+1} & = & l v_k+m(p-x_k)= -m x_k + l v_k + mp, \\
x_{k+1} & = &   x_k+ v_{k+1}= x_k + [-m x_k + l v_k + mp], \end{eqnarray}
which leads to
\begin{eqnarray}
v_{k+1}& = & -mx_k+lv_k+mp, \\
x_{k+1}& = & x_k+lv_k+mp-mx_k.
\end{eqnarray}

We can rewrite the above dynamic system in a matrix form as
\begin{equation}
Y_{k+1} = C Y_k+M p,
\end{equation}
where
\begin{equation}
Y_k=\begin{bmatrix}
x_k \\ v_k
\end{bmatrix}, \qquad
C=\begin{bmatrix}
1-m & l\\-m & l\\
\end{bmatrix}, \qquad
M=\begin{bmatrix}
m \\ m
\end{bmatrix}.
\end{equation}
Here, the $Y_k$ column vector corresponds to the states of positions and velocities of the bats at iteration $k$. Matrix $C$ is the dynamic matrix that governs the main properties of this dynamic system. $M$ is the input of the frequencies and $p$ is the current best solution in the system.

As the iterations continue and the bat population move towards $p$, the velocity of the bat population will approach zero. That is
\begin{equation} \lim_{k \rightarrow \infty} v_k=0. \end{equation}
Therefore, the final fixed point or point of convergence in the state space $Y_k$ is
\begin{eqnarray}
Y*=\begin{bmatrix}
p \\[5pt] 0
\end{bmatrix}.
\end{eqnarray}
The final state of convergence is that $\lim_{k \rightarrow \infty} x_k=p$ and
$\lim_{k \rightarrow \infty} v_k=0$ if there is no perturbation.

\subsection{Algorithm Convergence and Parameter Selection}
The main properties of the dynamic system is now determined by
the eigenvalues of the dynamic matrix $C$. That is
\begin{equation}
\det\left|
\begin{matrix}
1-m  -\lambda & l \\ -m & l -\lambda
\end{matrix}\right| =0,
\end{equation}
which gives
\begin{equation}
(1-m -\lambda) (l -\lambda) + ml=0,
\end{equation}
or simply
\begin{equation}
\lambda^2 + \lambda (m-l-1) + l=0.
\end{equation}
Thus, their solutions are
\begin{equation}
\lambda=\frac{-(m-l-1) \pm \sqrt{(m-l-1)^2-4l}}{2},
\end{equation}
which gives two eigenvalues $\lambda_1$ and $\lambda_2$.
For the dynamic system to be stable, their modules must be smaller than one $|\lambda| \le 1$. From Vieta's formulas for polynomials,
we know that $\lambda_1 \cdot \lambda_2 = l$
whose modulus should also be less than one, so we have $|l| \le 1$ or $-1 \le 1 \le 1$.

In addition, Vieta's formulas also indicate that
\begin{equation} \lambda_1+\lambda_2=-(m-l-1)=l-m+1 \le 2, \end{equation}
which must be less than 2 (i.e., $\lambda_1+\lambda_2 \le 2$) so that
each eigenvalue is potentially less than 1. We have
\begin{equation} l \le m+1,  \label{lam-eq-40} \end{equation}

For the conditions that the modulus of the biggest eigenvalue must be smaller than one $|\lambda| \le 1$, we have
\begin{equation}
\frac{(l-m+1) \pm \sqrt{(m-l-1)^2-4l}}{2} \le +1, \label{lam-eq-100} \end{equation}
or
\begin{equation} -1 \le \frac{(l-m+1) \pm \sqrt{(m-l-1)^2-4l}}{2}. \label{lam-eq-200}
\end{equation}
The first equation (\ref{lam-eq-100}) becomes
\begin{equation} (l-m-1) \le  \mp \sqrt{(m-l-1)^2 -4l}, \end{equation}
or \begin{equation} (l-m-1)^2 \ge (m-l-1)^2 -4l, \end{equation}
which gives
\begin{equation} (l-m-1)^2 - [(m-1-1)^2-4l] = 4m \ge 0, \end{equation}
or simply $m \ge 0$. Here, we have used the fact $l \le m+1$ (or $l-m-1 \le 0$),
thus the inequality should be reversed when taking the square.

Similarly, the other condition becomes
\begin{equation} -(l-m+3) \le \pm \sqrt{(m-l-1)^2 - 4l}, \end{equation}
or \begin{equation} (l-m+3)^2 \ge (m-l-1)^2 - 4l, \end{equation}
which gives \begin{equation}
(l-m+3)^2 - [(m-l-1)^2-4l] =4 (2l+2 -m) \ge 0, \quad
\textrm{ or } \;\;\; 2l+2 \ge m.  \end{equation}
Therefore, the conditions for stability and convergence are
\begin{equation}\begin{cases}
-1\le l \le +1, \\
m \ge 0, \\
2l-m+2 \ge 0.
\end{cases}
\end{equation}
which form a triangular region in the parameter space of $l$ and $m$, as shown in Fig.~\ref{fig-100}.
\begin{figure}[!htb]
  \centering
  \includegraphics[width=4.5in,height=2.5in]{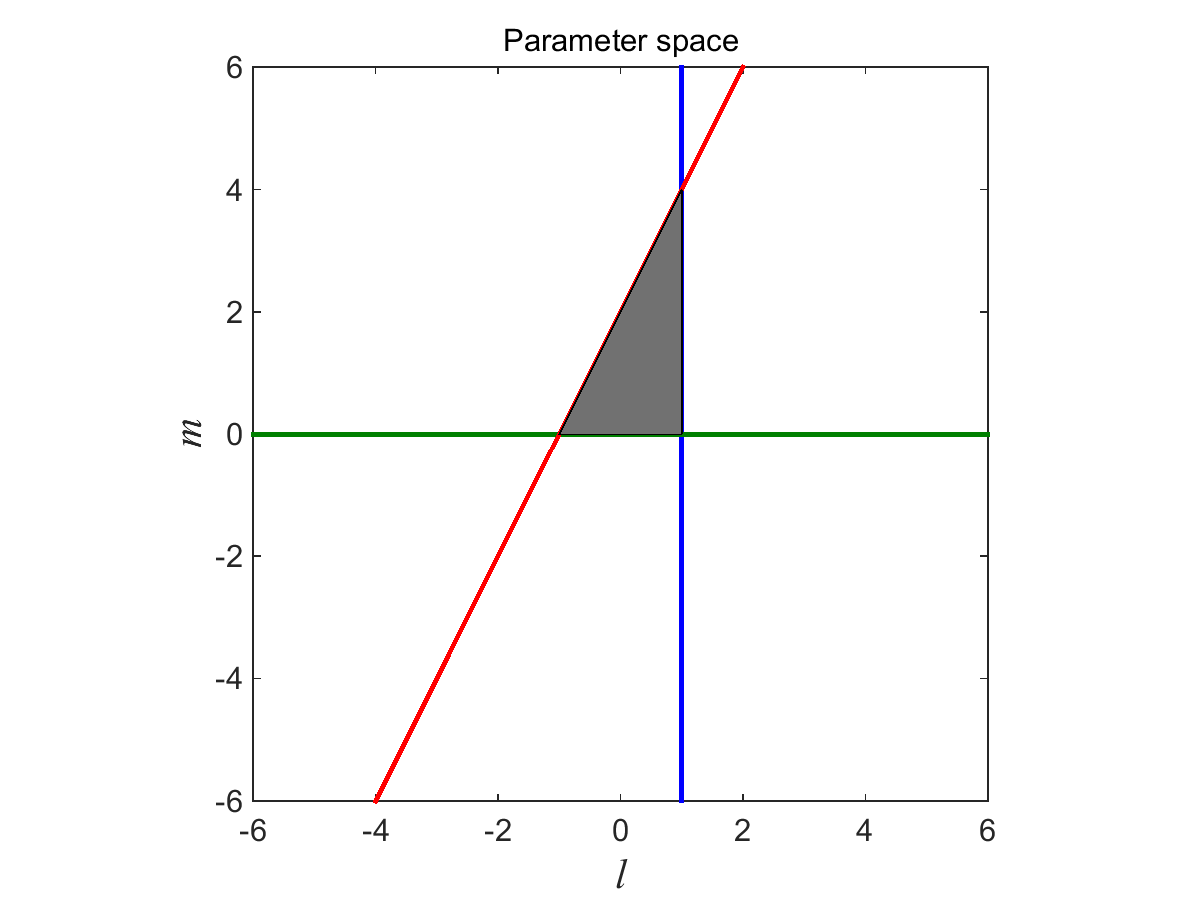}
  \caption{Parameter ranges for the bat algorithm to be stable. \label{fig-100}}
\end{figure}
The above analysis shows that within the parameter ranges of $m$ and $l$, the bat algorithm will not only converge towards the optimality, it will also converge stably. In this case, the algorithm will behave stable and converge quickly in practice. However, it should be emphasized that the dynamic model presented in this paper has not considered the variation of pulse emission rate and loudness, thus the actual parameter ranges may be different from the above results. Even so, this simplified model has enabled us to understand the influence of parameter values for the bat algorithm.

In the rest of the paper, we will use some selected benchmarks to show such convergence properties.

\section{Validation by Numerical Experiments}
In order to verify the BA algorithm and show the convergence characteristics discussed above in this paper, we have conducted some numerical experiments using
a few selected benchmark functions with very difficult properties and modalities.
These functions are listed in Table 1 where the dimensionality is chosen as $D=30$ for all functions.
{\small
\begin{table}[!htb]
\centering
\caption{simulation benchmarks}
\begin{tabular}{c c c c c}
\hline
Function name & Functions & Ranges & $f_{\min}$\\
\hline
Sphere & $f_1(\vec x)=\overset {n}{\underset {i=1}{\sum}}
x_i^2$  & [-5.12,5.12] & 0\\
Griewank & $f_2(\vec x)=\overset {n}{\underset {i=1}{\sum}} \frac {x_i^2}4000-\overset {n}{\underset {i=1}{\prod}} \cos(\frac{x_i}{\sqrt{i}})+1$  & [-600,600] & 0\\
Schwefel & $f_3(\vec x)=\overset {n}{\underset {i=1}{\sum}} |x_i|+\prod _{i=1}^n |x_i| $  &[-10,10] & 0\\
Quartic & $f_4(\vec x)=\overset {n}{\underset {i=1}{\sum}} ix_i^4+rand[0,1]$  & [-100,100] & 0\\
Rosenbrock & $f_5(\vec x)=\overset {n-1}{\underset {i=1}{\sum}}[(x_i-1)^2+100(x_{i+1}-x_i^2)^2]$  & [-5,5] & 0\\
Yang & $f_6(\vec x)=(\overset {n}{\underset {i=1}{\sum}} |x_i| exp[-\overset {n}{\underset {i=1}{\sum}} \sin (x_i^2)])$  & [-2$\pi$,2$\pi$] & 0\\
Zakharov & $f_7(\vec x)=\overset {n}{\underset {i=1}{\sum}} x_i^2+(\overset {n}{\underset {i=1}{\sum}} \frac {ix_i}2)^2+(\overset {n}{\underset {i=1}{\sum}} \frac {ix_i}2)^4$  & [-5,5] & 0\\
Step Function & $f_8(\vec x)=\overset {n}{\underset {i=1}{\sum}} (\lfloor x_i+0.5 \rfloor)^2$  & [-100,100] & 0\\
Rastrigin & $f_9(\vec x)=\overset {n}{\underset {i=1}{\sum}} (x_i^2-10\cos(2 \pi x_i)+10)$ & [-5.12,5.12] & 0\\
\hline
\end{tabular}
\end{table}
}

For each function, the bat algorithm has been executed with a maximum number of iterations $t_{\max}=500$ with a population size $n=12$, $m=2$ and $l=0.5$. The dimensions for all functions are $D=30$.

The convergence plots for all the functions are shown in Fig.~\ref{fig-200}.
\begin{figure}[!htb]
  \centering
  \includegraphics[width=4in,height=3in]{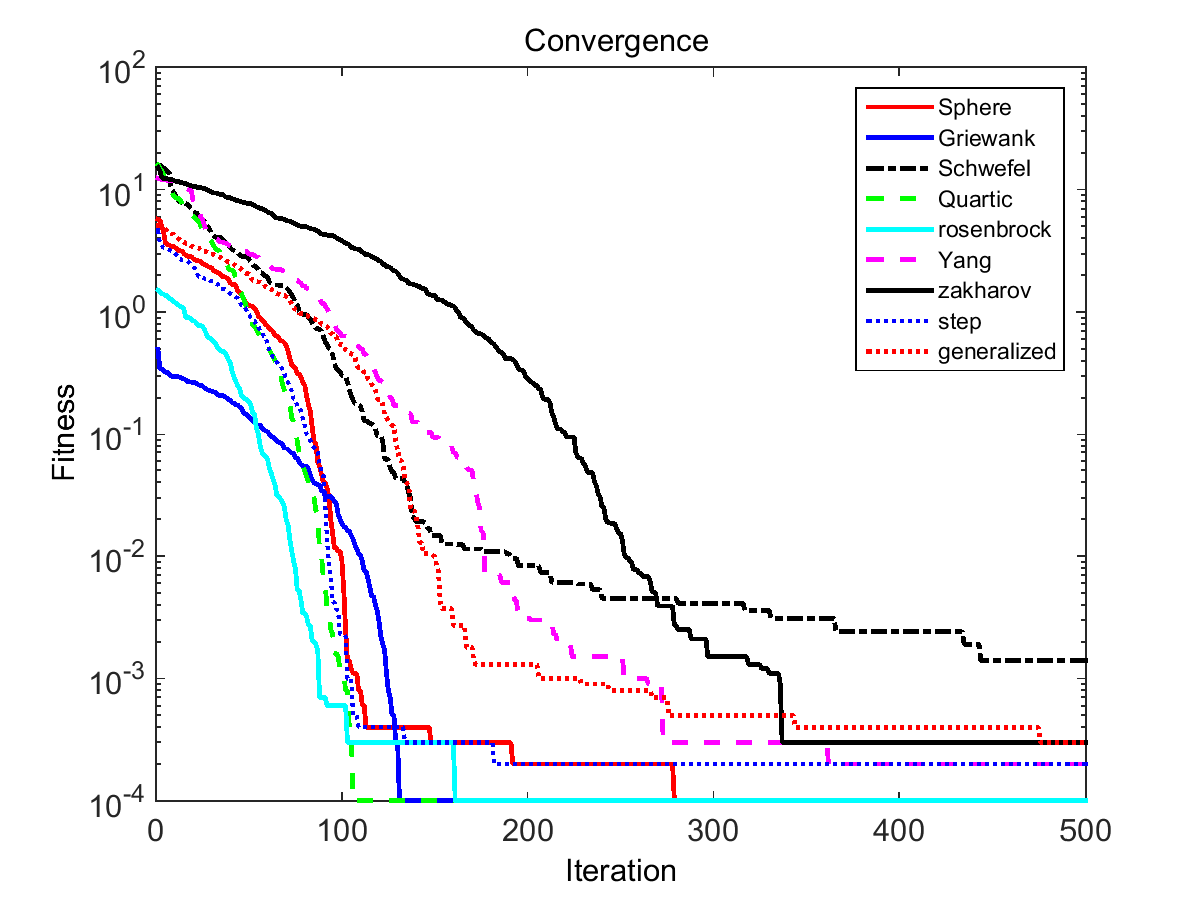}
  \caption{Plots of convergence for various functions. \label{fig-200} }
\end{figure}
It can be observed clearly from this figure that all functions can converge
quickly, specially at the early stage of the iterations. However, if the parameter ranges lie outside the stable domain, the rate of convergence can be significantly lower, and the very slow convergence or even premature convergence can occur
as can see from Fig.~\ref{fig-300} where $m=-3$ and $l=4$ are used, even though all the other parameters remain the same.
\begin{figure}[!htb]
  \centering
  \includegraphics[width=4in,height=3in]{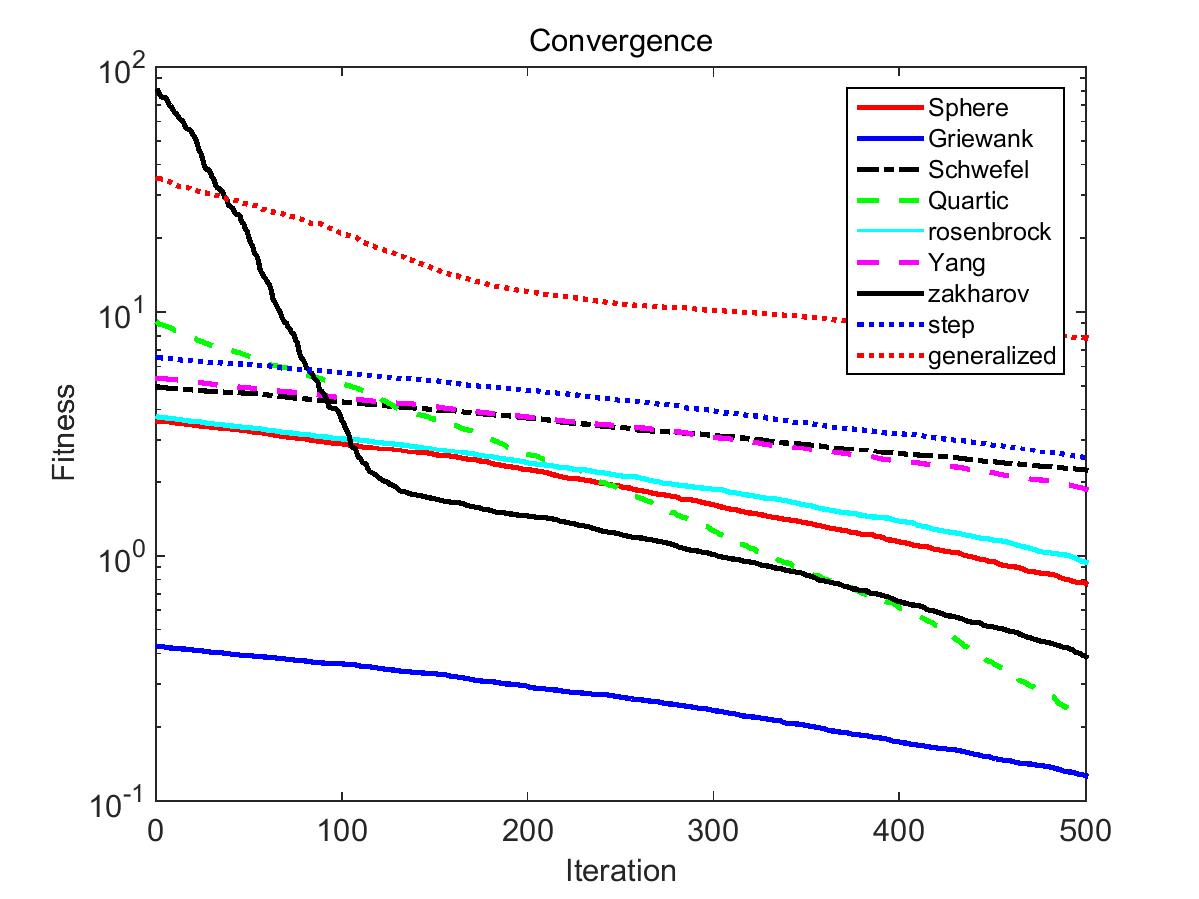}
  \caption{Plot of convergence when the parameters lie outside the triangular domain. \label{fig-300} }
\end{figure}

All the above have demonstrated that the algorithm can converge both quickly and robustly. Thus, the algorithm can be suitable for difficult optimization problems where
optimal or nearly optimal solutions are needed quickly.

\section{Conclusions}

The bat algorithm has been shown to be effective in practice, but there is not theoretical analysis in the literature. This paper provides some theoretical analysis of the standard bat algorithm using both a simplified Markov chain model and a dynamic matrix model. The Markov model shows that the algorithm can converge to the global optimality with probability one as the number of iterations becomes sufficiently large. The dynamic model looks at the algorithm from a different perspective. By extending the models with more parameters, we have then obtained some insight why some parameters are not important, while others can be tuned. As a result, the parameter ranges of some key parameters have been identified.

Following the theoretical analyses, we have used some benchmark test functions to validate the bat algorithm using the appropriate parameters. Good convergence has been observed for all functions, which is consistent with the theoretical results.

It is worth pointing out that the models used in this paper are simplified models
without considering the variation of pulse emission rate and loudness. Future work will try to extend to investigate the effect of such factors in the convergence properties of the bat algorithm. In addition, even we now understand why the bat algorithm converge with a clear parameter region, it still lacks the information on the rate of convergence and how the parameter values will affect the rate of convergence. Future work will also investigate this issue further with more rigorous analyses.

\end{document}